\documentclass[12pt]{amsart}

\usepackage{epsfig}
\usepackage{amssymb}
\usepackage{amsthm}
 \usepackage{mathrsfs}  
 
\usepackage{geometry}            % See geometry.pdf to learn the layout options. There are lots.
\usepackage[parfill]{parskip}    % Activate to begin paragraphs with an empty line rather than an indent
\usepackage{graphicx}
\usepackage{fullpage}

\usepackage{float}
\usepackage{epstopdf}
\usepackage{extarrows}
\usepackage{enumerate}
\usepackage{chemarrow}

\usepackage{thmtools}
\usepackage{mathtools}
\usepackage{bbm}

\usepackage{tikz}
\usetikzlibrary{matrix}
\usetikzlibrary{shapes}

\usepackage[cmtip,arrow]{xy}
\usepackage{pb-diagram,pb-xy}

\usepackage{verbatim}
\theoremstyle{plain}
\newtheorem{theorem}{Theorem}[section]
\newtheorem{prop}[theorem]{Proposition}
\newtheorem{cor}[theorem]{Corollary}

\newtheorem{lem}[theorem]{Lemma}

\theoremstyle{definition}
\newtheorem{defn}[theorem]{Definition}

\newtheorem{rem}[theorem]{Remark}
\newtheorem{exm}[theorem]{Example}

\newcommand{\A}{\mathbb{A}}

\newcommand{\N}{\mathbb{N}}

\newcommand{\Q}{\mathbb{Q}}
\newcommand{\R}{\mathbb{R}}

\newcommand{\V}{\mathbb{V}}

\newcommand{\Z}{\mathbb{Z}}

\newcommand{\m}{\mathfrak{m}}

\newcommand{\p}{\mathfrak{p}}

\DeclareMathOperator{\Gal{Gal}}

\newcommand{\length}{\operatorname{length}}
\title{Monomial Valuations, Cusp Singularities, and Continued Fractions}
\author[Bruce]{Juliette Bruce}
\author[Logue]{Molly Logue}
\author[Walker]{Robert Walker}

\address{University of Michigan, Ann Arbor, MI, 48109}
\email{juliette.bruce@math.wisc.edu}

\address{University of Michigan, Ann Arbor, MI, 48109}
\email{mmlogue@umich.edu}

\address{Department of Mathematics, University of Michigan, Ann Arbor, MI, 48109}
\email{robmarsw@umich.edu}

\date{\today}                                           % Activate to display a given date or no date

\begin{document}

\begin{abstract}This paper explores the relationship between real valued monomial valuations on $k(x,y)$, the resolution of cusp singularities, and continued fractions. It is shown that up to equivalence there is a one to one correspondence between real valued monomial valuations on $k(x,y)$ and continued fraction expansions of real numbers between zero and one. This relationship with continued fractions is then used to provide a characterization of the valuation rings for real valued monomial valuations on $k(x,y)$. In the case when the monomial valuation is equivalent to an integral monomial valuation, we exhibit explicit generators of the valuation rings. Finally, we demonstrate that if $\nu$ is a monomial valuation such that $\nu(x)=a$ and $\nu(y)=b$, where $a$ and $b$ are relatively prime positive integers larger than one, then $\nu$ governs a resolution of the singularities of the plane curve $x^{b}=y^{a}$ in a way we make explicit. Further, we provide an exact bound on the number of blow ups needed to resolve singularities in terms of the continued fraction of $a/b$. \end{abstract}

%\begin{abstract}This paper explores the relationship between monomial valuations on $k(x,y)$, the resolution of cusp singularities, and continued fractions. It is shown that up to equivalence there is a one to one correspondence between monomial valuations on $k(x,y)$ and continued fraction expansions of rational numbers between zero and one. We exhibit explicit generators of the valuation rings for these monomial valuations. Further, we demonstrate that if $\nu$ is a monomial valuation such that $\nu(x)=a$ and $\nu(y)=b$ where $a$ and $b$ are relatively prime positive integers, then $\nu$ governs the resolution of the singularities of the plane curve $x^{b}=y^{a}$ in a way we make explicit. Further, we provide an exact bound on the number of blow ups needed in terms of the continued fraction of $b/a$. \end{abstract}

\thanks{The first and second authors were supported by NSF grant DMS-1001764 and NSF RTG grant DMS-0943832.The third author was supported by a NSF GRF under Grant No. PGF-031543.}

\maketitle

%%%%%%%%%%%%%%%%%%%%%%%%%%%%%%%BeginTextBellow%%%%%%%%%%%%%%%%%%%%%%%%%%%%%%%%%%%%%%%
%%%%%%%%%%%%%%%%%%%%%%%%%%%%%%%BeginTextBellow%%%%%%%%%%%%%%%%%%%%%%%%%%%%%%%%%%%%%%%
%%%%%%%%%%%%%%%%%%%%%%%%%%%%%%%BeginTextBellow%%%%%%%%%%%%%%%%%%%%%%%%%%%%%%%%%%%%%%%
%%%%%%%%%%%%%%%%%%%%%%%%%%%%%%%BeginTextBellow%%%%%%%%%%%%%%%%%%%%%%%%%%%%%%%%%%%%%%%

\section{Background}

Since the late 1930's, it has been known that valuation theory and the resolution of singularities are closely connected. In particular, it was found that the existence of a local uniformization of a valuation in essence provides a way to resolve a singularity locally.

Recall the local uniformization question asks: If we fix a ground field $k$, a field extension $K$, and a one-dimensional valuation $\nu$ on $K$, does there exists a complete model $X$ of $K$ such that the center of $\nu$ on $X$ is non-singular?  In 1940, Zariski showed that if $K$ is of dimension two and $k$ is algebraically closed and of characteristic zero, then such a model always exists \cite{zariski_local_1940}.  Utilizing this result, Zariski proved that a resolution of singularities for 3-folds is possible in characteristic zero \cite{zariski_compactness_1944}, \cite{zariski_reduction_1944}.  Since then this method of local uniformization has been extended and used to resolve the singularities of any varieties of dimension less than or equal to three over a field of arbitrary characteristic \cite{abhyankar_local_1956}, \cite{abhyankar_resolution_1966}, \cite{cossart_resolution_2008}, \cite{cossart_resolution_2009}, \cite{cutkosky_resolution_2009}. A nice treatment of the theory of resolutions of singularities can be found in \cite{cutkosky_resolution_2004} and \cite{kollar_lectures_2007}. Of course Hironaka's famous theorem on resolution of singularities for varieties over a field of characteristic zero does not use valuation theory, but it is possible that an eventual solution to the characteristic $p$ case will, since most significant progress in characteristic $p$ uses valuation theory \cite{hironaka_resolution_1964-1}, \cite{hironaka_resolution_1964}. 

There has been interesting work exploring how continued fractions are related to resolving various singularities.  For example, it has been shown that the minimal resolutions of toric singularities can be calculated by Hirzebruch-Jung continued fractions \cite{cox_toric_2011}.  A nice overview of the numerous ways continued fractions have come up in the studies of singularities, including the case of plane curve singularities, is found in \cite{popescu-pampu_geometry_2005}.   

There has been work connecting continued fractions and valuation theory.  Recall that given a valuation $\nu$ on $k[[x,y]]$, we can associate $\nu$ with a sequence of infinitely nearby points.  In fact, there is a bijection between the set of valuations up to equivalence and the set of infinitely nearby points \cite{favre_valuative_2004}, \cite{zariski_commutative_1960}.  If $\nu$ is a monomial valuation, it turns out this sequence of infinitely nearby points is given by the continued fraction of $\nu(x)/\nu(y)$ \cite{favre_valuative_2004}, \cite{spivakovsky_valuations_1990}.

Many of our results are, in fact, scattered throughout the literature in some form. Specifically see \cite{favre_valuative_2004}, \cite{popescu-pampu_geometry_2005}, \cite{spivakovsky_valuations_1990}.  However, we have not found these results presented in a unified, cohesive, and elementary fashion. Additionally, many of our proofs of these results differ from those found in the existing literature. In particular, by sacrificing some level of generality, the proofs we present are more elementary and completely self-contained. 

%%%%%%%%s%%%%%%%%%%%%%%%%%%%%%%%%%%%%%%%%%%%%

\section{Continued Fractions}\label{section: CF}

Let us recall a few basic definitions and facts regarding continued fractions and fix notations.  A more complete treatment of the theory of continued fractions may be found in \cite{collins_continued_????} and Chapter 7 of \cite{strayer_elementary_2002}.

Fix $d_0 \in \Z$ and $d_1, d_2, \ldots \in \Z^+$. A \textbf{finite continued fraction} is an expression of the form
$$[d_0 ; d_1, d_2, \ldots, d_k] = d_0 + \cfrac{1}{d_1+\cfrac{1}{d_2+\cfrac{1}{\ddots +\cfrac{1}{d_k}}}}.$$ 
We can generalize this notion by defining an \textbf{infinite continued fraction} as the limit of a sequence of finite continued fractions, which we signal with the notation $$[d_0 ; d_1, d_2, \ldots] := \lim_{k \to \infty} [d_0 ; d_1, d_2, \ldots, d_k].$$
Due to the following well-known fact, which we leave to the reader to check, we really need not worry about the convergence of infinite continued fractions. 

\begin{prop}Given a sequence of integers $\{d_{i}\}_{i=0}^{\infty}$ where the $d_{i}$ positive for all $i\neq 0$ the continued fraction $[d_{0}; d_{1}, d_{2},\ldots]$ exists.\end{prop}
%The finite continued fractions in question are called the \textbf{convergents} of $[d_0 ; d_1, d_2, \ldots]$.  ]

Every real number $r$ can be expressed as a continued fraction, which is essentially unique.  To see why this is the case, we provide the following algorithm for constructing a continued fraction for $r$.  Setting $r_{0}=r$ and $d_{0}=\lfloor r_{0}\rfloor$, where $\lfloor r_{0}\rfloor$ is the integer part of $r_{0}$, we then define $d_{n}$ and $r_{n}$ recursively by
\begin{equation*}d_{n}=\lfloor r_{n}\rfloor \quad \quad \text{and}\quad \quad r_{n}=\frac{1}{r_{n-1}-d_{n-1}},\end{equation*}
so long as $r_{n-1}-d_{n-1}\neq0$.

If $r_{n}-d_{n}=0$ for some $n$, we set $k=n$, and then $r=[d_{0}; d_{1}, d_{2},\ldots, d_{n}]$.  Otherwise, this process does not terminate, and so $r$ is the infinite continued fraction $[d_0 ; d_1, d_2, \ldots]$.  Of course, we must be careful that the limit of the convergents actually exists and is $r$, but this is a rather simple exercise we shall leave to the reader.  If $r$ has an infinite continued fraction expansion, then by induction we find that this expansion is in fact unique.  If, however, the continued fraction expansion of $r$ is finite, then the expansion is unique up to replacing $d_k$ with the pair $d_{k}-1, 1$ in the bracket notation, i.e., $[0; 1, 2] = [0; 1 , 1, 1]$.

In illustration, consider the following example. 

\begin{exm}\label{ExmConFrac}Let us compute the continued fraction expansion of $3/2$.  Applying the above algorithm, we have that 
\begin{align*}
3 &= 1\cdot 2+1\\
2 &= 2\cdot1+0,
\end{align*}
and so $3/2=[1;2]=[1;1,1]$.  By contrast, the continued fraction expansion of the irrational number $\frac{1}{\pi}$ is 
$$\frac{1}{\pi} = [0; 3, 7, 15, 1, 292, 1, 1, 1, 2, 1, 3, 1, 14, 2, 1, 1, 2, 2, 2, 2,1,84...].$$ 
\end{exm}

\begin{prop}
A real number $r$ has a finite continued fraction expansion if and only if $r$ is rational.
\end{prop}
\begin{proof}
If $r$ equals $[d_{0}; d_{1},d_{2},\ldots,d_{k}]$ then $r \in \Z \left(\frac{1}{d_1}, \ldots , \frac{1}{d_k}\right) \subseteq \Q$. 
Conversely, let $r = a/b \in \Q$. If $b=1$ then $a/b$ can be represented as a continued fraction by $[a]$.  Therefore, suppose $b\neq1$.  By the division algorithm, there exist integers $q_{0}$ and $p_{0}$ such that $a = q_0 b + p_0$ where $0 \le p_0 < b$.  Notice that $\lfloor a/b\rfloor$ equals $q_{0}$ and so $d_{0}=q_{0}$.  Following the algorithm described above, and utilizing the fact that $q_0=a/b-p_{0}/b$, we have that 
$$ d_{1}=\left\lfloor \frac{1}{r_{0}-d_{0}}\right\rfloor= \left\lfloor \frac{1}{\frac{a}{b}-q_{0}}\right\rfloor = \left\lfloor \frac{b}{p_{0}}\right\rfloor.$$
Applying the division algorithm, there exist $p_{1},q_{1}\in \Z$ such that $b=q_{1}p_{0}+p_{1}$ where $0\leq p_{1}<p_{0}$.  Therefore, we find that $d_{2}$ is equal to 
$$d_{2}=\left\lfloor \frac{1}{r_{1}-d_{1}}\right\rfloor= \left\lfloor \frac{1}{\frac{b}{p_0}-q_{1}}\right\rfloor = \left\lfloor \frac{p_0}{p_{1}}\right\rfloor.$$
By induction, we have that 
$$ d_{n}=\left\lfloor \frac{p_{n-2}}{p_{n-1}}\right\rfloor ,$$
where $a = p_{-2}, b = p_{-1}$, and $p_{n-3}=q_{n-2}p_{n-2}+p_{n-1}$ with $0\leq p_{n-1}<p_{n-2}$.  Therefore, the $p_{i}$'s form a monotone strictly decreasing sequence of non-negative integers.  So for some $k$, we have $p_{k}=0$ implying that $d_{k+2}$ must equal zero.  Thus, the continued fraction of $a/b$ is finite.
\end{proof}

\begin{rem}
The argument for the reverse direction of the proposition is essentially an application of the Euclidean algorithm.  In particular, if $a$ and $b$ are positive integers and $q_{0}, q_{1}, \ldots, q_{k}$ are the quotients found when the Euclidean algorithm is used to calculate the greatest common divisor of $a$ and $b$ such that $a=q_{0}b+r_{0}$, then the continued fraction expansion of $a/b$ is $[q_{0}; q_{1}, q_{2},\ldots, q_{k}]$.
\end{rem}

%%%%%%%%s%%%%%%%%%%%%%%%%%%%%%%%%%%%%%%%%%%%%

\section{Classifying $\R$-valued Monomial Valuations on $k(x,y)$}
In this section, our goal is to provide an explicit description of the valuation ring $R_{\nu}$ of a given $\R$-valued monomial valuation $\nu$ on a field extension $K=k(x,y)$ of a field $k$. We give an answer for when $\nu$ is rational, and provide a way of thinking about general monomial valuations by introducing the concept of the valuation tree.\footnote{This is not the same as the valuative tree defined in \cite{favre_valuative_2004}.}  Additionally, by using these valuation trees, we establish a correspondence between $R_{\nu}$ and the continued fraction expansion of $\nu(x)/\nu(y)$.  We first review the basics of valuation theory.  A more complete treatment of valuation theory may be found in \cite{matsumura_commutative_1989}, \cite{zariski_commutative_1960}. \\

Let $\Gamma$ be a totally ordered abelian group. Then a $\Gamma$-valued valuation is defined as follows:
\begin{defn}\label{DefVal}
Fix a ground field $k$ and a field extension $K$. A valuation on $K$ with respect to $k$, or $k$-valuation, is a map $\nu: K \rightarrow \Gamma \cup \{\infty\}$ satisfying:
\begin{itemize}
\item{The restriction of $\nu$ to $K^{\times}$ is a group homomorphism $K^{\times} \rightarrow \Gamma$}
\item{ $\nu(\lambda)=0,  \forall \lambda \in k^{\times}$}
\item $\nu(f)=\infty$ if and only if $f=0$
\item{ $\nu(f+g) \geq \min \{\nu(f), \nu(g) \}, \forall f, g \in K$.}
\end{itemize}
\end{defn}

We focus on valuations on the field $k(x,y)$ with respect to the subfield $k$.  Additionally, unless otherwise stated we shall take $\Gamma=\R$.  Note that for any valuation $\nu$ on $k(x,y)$, 
$$\nu(x^iy^j)=i\nu(x)+j\nu(y),$$ 
since $\nu$ is a group homomorphism. This means that the value of any monomial in $k(x,y)$ is determined solely by $\nu(x)$ and $\nu(y)$.  We are particularly interested in valuations in which the value of any polynomial is determined by the values of $x$ and $y$. These are called monomial valuations \cite{favre_valuative_2004}.

\begin{defn}
The monomial valuation, $\nu$, on $k(x,y)$ is defined by $\nu(x)=a, \nu(y)=b$, where $a, b \in \R_{\geq 0}$, and 
\[ \nu \left( \sum_{i,j} \lambda_{ij}x^iy^j \right)= \min_{i,j}\{ ia+jb; \lambda_{ij} \neq 0 \}.\]
\end{defn}

Note that since $\nu(fg^{-1})=\nu(f)-\nu(g)$, for $f, g \in k[x,y]$, defining $\nu$ for polynomials $\sum \lambda_{ij}x^iy^j$ also determines $\nu$ for rational expressions.  Thus, monomial valuations on $k(x,y)$ are determined solely by the values of $x$ and $y$.

By changing the values of $\nu(x)$ and $\nu(y)$ in the monomial valuation we obtain different valuations, however, some valuations are essentially the same.  For example, strictly speaking the monomial valuations determined by $\nu_{1}(x)=1$, $\nu_{1}(y)=2$ and $\nu_{2}(x)=.5$, $\nu_{2}(y)=1$ are different, but for any $f\in k(x,y)$ we have the relation $\nu_{1}(f)=2\nu_{2}(f)$. These valuations are essentially the same. Formally:

\begin{defn} Two valuations $\nu_1$ and $\nu_2$ on $k(x,y)$ are said to be equivalent if there exists $\lambda \in \R_{>0}$ such that $\nu_1(f)=\lambda \nu_2(f), \forall f \in k(x,y)$.
\end{defn}

This allows us to make certain standardizations. For this paper, we will consider $\R$-valued valuations $\nu$ on $k(x,y)$ such that $\nu(x)>\nu(y)>0$. If a valuation does not have this property, we can relabel $x$ and $y$ or multiply by a scalar to obtain an equivalent valuation that does have this property. Furthermore, we may also scale by $1/\nu(y)$ so as to assume $\nu(x)>\nu(y)>0$ and $\nu(y)=1$ if we like. 

\begin{defn}
Given a $k$-valuation $\nu$ on $k(x,y)$, the valuation ring is given by 
$$R_{\nu} := \{ f \in k(x,y) : \nu(f) \geq 0 \}.$$
\end{defn}

It follows directly from the axioms of a valuation that $R_{\nu} $ is a ring. This ring $R_{\nu}$ is a sub-$k$-algebra of $k(x,y)$, and is a local ring with maximal ideal 
$$\m_{\nu}:=\{f \in k(x,y) : \nu(f)>0\}.$$

\begin{prop} Two valuations $\nu_1, \nu_2$ on $k(x,y)$ are equivalent if and only if $R_{\nu_1}=R_{\nu_2}$.\end{prop}

Valuation rings have many additional characterizations and properties. For a more complete treatment see \cite{matsumura_commutative_1989}.  This proposition gives us a nice way to approach the problem of classifying the monomial valuations on $k(x,y)$.  Namely, to understand monomial valuations on $k(x, y)$ it suffices to determine the valuation rings.  

\subsection{Valuation Trees}\label{Val-Tree}
While the definition of a valuation ring gives an abstract definition of these objects, we would like to be able to provide a more explicit description of them. Given a valuation $\nu$ on $k(x,y)$ generated by $\nu(x)=a, \nu(y)=b$ where $a,b\in \R$, what is $R_{\nu}$? To find this ring, we need to determine the largest subring $R \subset k(x,y)$ whose elements all have non-negative value.  Heuristically, one might try and find the valuation ring by carefully adding fractions of elements in $k[x,y]$ such that elements of the resulting ring still have non-negative value.

To make this heuristic process formal, we introduce the notion of the valuation tree for $k(x,y)$. The valuation tree, $T$, is a directed graph whose vertices are certain subrings of $k(x,y)$. The ring $k[x,y]$ is a vertex of $T$, and then we define the vertex set of $T$ recursively.  In particular, if $k[f,g]$ is a vertex of $T$ then $k[f, g/f]$ and $k[g, f/g]$ are also vertices of $T$, and there is an edge going from $k[f,g]$ to each of $k[f, g/f]$ and $k[g, f/g]$. 
This definition gives us the following infinite tree:

\begin{figure}[H]
\begin{center}
 	\begin{tikzpicture}
	[scale=.75,colorstyle/.style={circle, draw=black!100,fill=black!100, thick,
	inner sep=0pt, minimum size=2 mm}]
	\node at (0,0){$k[x,y]$};
	\node at (5,3){$k[y, x/y]$};
	\node at (5,-3){$k[x, y/x]$};
	
	\node at (10, 4.5){$k[y, x/y^{2}]$};
	\node at (10, 1.5){$k[x/y, y^{2}/x]$};
	\node at (10, -1.5){$k[y/x, x^{2}/y]$};
	\node at (10, -4.5){$k[x, y/x^{2}]$};
	
	\node at (16.4, 2.25){$k[y^2/x, x^2/y^3]\cdots$};
	\node at (16.3, .75){$k[x/y, y^3/x^2]\cdots$};
	\node at (16.3, 3.5){$k[x/y^2, y^3/x]\cdots$};
	\node at (16, 5){$k[y,x/y^3]\cdots$};
	\node at (16.25,-.5){$k[y/x,x^3/y^2]\cdots$};
	\node at (16.3,-2.5){$k[x^2/y, y^2/x^3] \cdots$};
	\node at (16.25, -3.75){$k[y/x^2, x^3/y] \cdots$};
	\node at (16, -5.5){$k[x,y/x^3]\cdots$};
	
	\draw[->,thick](.8,.1)--(3.8,3);
	\draw[->,thick](.8,-.1)--(3.8,-3);
	\draw[->,thick](6.1,-2.9)--(8.5, -1.5);
	\draw[->,thick](6.1,2.9)--(8.4, 1.5);
	\draw[->,thick](6.1,3.1)--(8.65,4.5);
	\draw[->,thick](6.1,-3.1)--(8.4,-4.5);
	
	\draw[->,thick](11.5,1.6)--(14.3,2.25);
	\draw[->,thick](11.5,1.4)--(14.2,.75);
	\draw[->,thick](11.5,4.5)--(14.2,5);
	\draw[->,thick](11.5,4.50)--(14.2,3.5);
	\draw[->,thick](11.5,-1.5)--(14.2,-.5);
	\draw[->,thick](11.5,-1.5)--(14.2,-2.5);
	\draw[->,thick](11.5,-4.5)--(14.2,-3.75);
	\draw[->,thick](11.5,-4.5)--(14.2,-5.5);
	
	%\draw[dashed](17.6,2.25)--(18.5,2.75);
	%\draw[dashed](11.2,1.5)--(12.5,1);
	
	%\draw[dashed](11.2,-1.5)--(12.5,-1);
	%\draw[dashed](11.2,-1.5)--(12.5,-2);
	%\draw[dashed](11.5,-4.5)--(12.5,-5);
	%\draw[dashed](11.5,-4.5)--(12.48,-4);
	
	\end{tikzpicture}
\end{center}
 \caption{Valuation tree, $T$}
\end{figure}
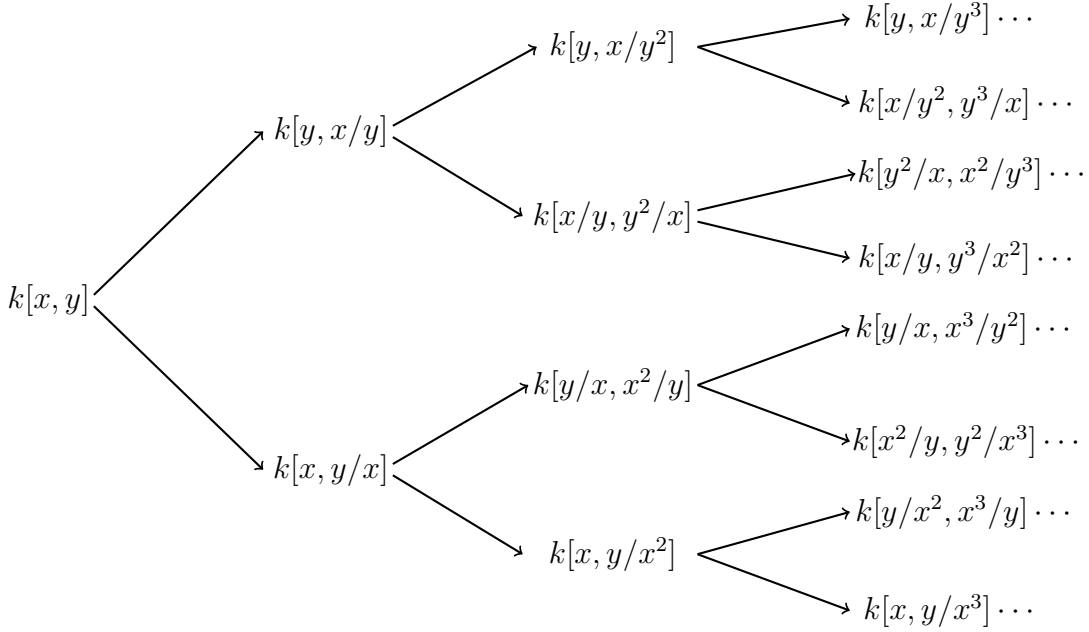
All of the rings in $T$ are subrings of $k(x,y)$. Given a $k$-valuation $\nu$ on $k(x,y)$, certain rings in $T$ will also be subrings of $R_{\nu}$. Determining these subrings will help us determine what $R_{\nu}$ is. To find these subrings, we must introduce the following definition. 

\begin{defn}Let $\nu$ be a $k$-valuation on $k(x,y)$. Then a ring $k[s_{1}, s_{2}, \ldots, s_{n}]\subset k(x,y)$ is positive with respect to $\nu$, or $\nu$-positive, if $\nu(s_{i})>0$ for all $s_{i}$.\end{defn}

Now that we have this definition, we can define the positive path of $T$ for a valuation $\nu$.

\begin{defn}
Given a monomial $k$-valuation on $k(x,y)$, $P_{\nu}$ is the subgraph of $T$ whose vertices are all positive with respect to $\nu$.
\end{defn}

If we fix a $k$-valuation, $\nu$, on $k(x,y)$, then by our convention $\nu(x)>\nu(x)>0$, so $k[x,y]$ is positive and therefore is in $P_{\nu}$. If $k[s,t]$ is a positive vertex in $T$ such that $\nu(s)\neq\nu(t)$ then either $\nu(s/t)<0$ or $\nu(t/s)<0$ and so only one of $k[s,t/s]$ or $k[t,s/t]$ is positive with respect to $\nu$.  Thus, $P_{\nu}$ is not just any subgraph, but is in fact a path within $T$.

Note that each ring in $P_{\nu}$ contains the previous ring in the path. So we obtain a chain of increasingly larger and larger rings whose elements all have positive value under $\nu$. In order to see this concretely, we consider the following example.

\begin{exm}\label{Val23}If $\nu$ is the monomial valuation on $k(x,y)$ such that $\nu(x)=3$ and $\nu(y)=2$, recall that $k[x,y]$ is $\nu$-positive. Looking at the next two vertices in the tree, we see that $k[x, y/x]$ is not positive while $k[y,x/y]$ is positive. So $k[y,x/y]$ is in $P_{\nu}$. Since $\nu(x/y^2)=-1$, $k[y, x/y^{2}]$ is not in $P_{\nu}$, and since $\nu(y^2/x)=1$, $k[x/y, y^{2}/x] \in P_{\nu}$, and this is where the path ends since neither $k[x/y,y^3/x^2]$ nor $k[y^2/x, x^2/y^3]$ lies in $P_{\nu}$, as $\nu(y^3/x^2)=\nu(x^2/y^3)=0$. Therefore, $P_{\nu}$ is the bolded path in the figure below: 
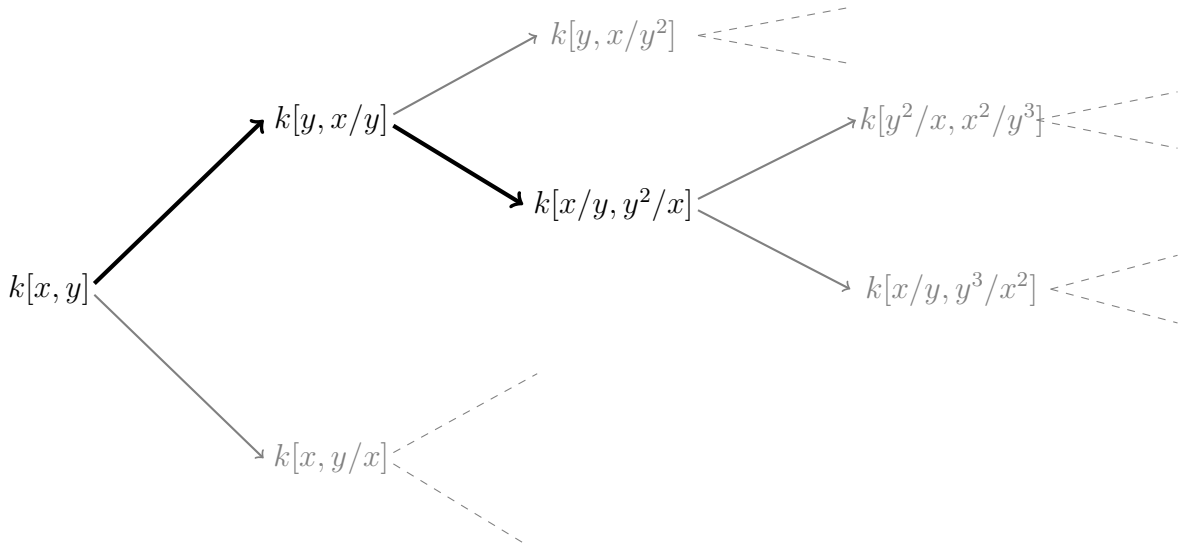
\begin{figure}[H]
\begin{center}
 	\begin{tikzpicture}
	[scale=.75,colorstyle/.style={circle, draw=black!100,fill=black!100, thick,
	inner sep=0pt, minimum size=2 mm}]
	\node at (0,0){$k[x,y]$};
	\node at (5,3){$k[y, x/y]$};
	\node[gray] at (5,-3){$k[x, y/x]$};
	
	\node[gray] at (10, 4.5){$k[y, x/y^{2}]$};
	\node at (10, 1.5){$k[x/y, y^{2}/x]$};
	%\node at (10, -1.5){$\C[x, y/x^{2}]$};
	%\node at (10, -4.5){$\C[y/x, x^{2}/y]$};
	
	\node[gray] at (16, 3){$k[y^2/x, x^2/y^3]$};
	\node[gray] at (16, 0){$k[x/y, y^3/x^2]$};
	
	\draw[->,ultra thick](.8,.1)--(3.8,3);
	\draw[->,thick,gray](.8,-.1)--(3.8,-3);
	\draw[dashed,gray](6.1,-2.9)--(8.65, -1.5);
	\draw[->,ultra thick](6.1,2.9)--(8.4, 1.5);
	\draw[->,thick, gray](6.1,3.1)--(8.65,4.5);
	\draw[dashed,gray](6.1,-3.1)--(8.4,-4.5);
	
	\draw[->,thick,gray](11.5,1.6)--(14.3,3);
	\draw[->,thick,gray](11.5,1.4)--(14.2,0);
	\draw[dashed,gray](11.5,4.5)--(14.2,5);
	\draw[dashed,gray](11.5,4.5)--(14.2,4);
	
	\draw[dashed,gray](17.5,3)--(20,3.5);
	\draw[dashed,gray](17.5,3)--(20,2.5);
	\draw[dashed,gray](17.75,0)--(20,-.6);
	\draw[dashed,gray](17.75,0)--(20,.6);
	
	%\draw[dashed](11.2,1.5)--(12.5,2);
	%\draw[dashed](11.2,1.5)--(12.5,1);
	
	%\draw[dashed](11.2,-1.5)--(12.5,-1);
	%\draw[dashed](11.2,-1.5)--(12.5,-2);
	%\draw[dashed](11.5,-4.5)--(12.5,-5);
	%\draw[dashed](11.5,-4.5)--(12.48,-4);

	\end{tikzpicture}
\end{center}
 \caption{$P_{\nu}$ for monomial valuation given by $\nu(x)=3$ and $\nu(y)=2$.}
\end{figure}
\end{exm}
While in the above example $P_{\nu}$ is a finite sub-graph of $T$, this is not true in general, and later in this section we shall establish conditions needed on $\nu$ for $P_{\nu}$ to be finite.
 
Thus, we now have a way to associate each $\R$-valued monomial valuation to a path within $T$. Note, however, that not every path within $T$ corresponds to an $\R$-valued monomial valuation.  Consider, for example, the infinite path whose vertices are of the form $k[y, x/y^{t}]$ for all $t\in \N$.  If this path were to correspond to an $\R$-valued monomial valuation, then $\nu(x/y^{t})>0$ for all $t\in \N$, meaning $\nu(x)=\infty$, which cannot occur by our definition of valuation.  So this path does not correspond to an $\R$-valued monomial valuation.  However, this path does correspond to a $\Z\oplus\Z$-valued valuation on $k(x,y)$. See Section 3.4 and especially Example 3.16.

\subsection{Integral and Rational Monomial Valuations}
 
How does $P_{\nu}$ help us find the valuation ring?  Appealing again to the above heuristic, and since edges are, in essence, induced by inclusions, the terminal object in $P_{\nu}$, if it exists, is the largest subring in $k(x,y)$ positive with respect to $\nu$.  Thus, we would hope that this terminal ring in $P_{\nu}$ is in fact the valuation ring associated to $\nu$.  However, looking back at Example \ref{Val23}, we see that there are two issues with this hope.  First, while $k[x/y, y^{2}/x]$ is the terminal object in $P_{\nu}$, both $k[x/y, y^{3}/x^{2}]$ and $k[y^{2}/x, x^{2}/y^{3}]$ contain the terminal object, and neither contains an element with negative value.  Second, recall that $R_{\nu}$ is a local ring, while $k[x/y, y^{2}/x]$ is not local.

It turns out these issues are not hard to overcome.  In fact, in Example \ref{Val23} the valuation ring associated to the monomial valuation given by $\nu(x)=3$ and $\nu(y)=2$ is 
 $$k\left[\frac{x}{y}, \frac{y^{3}}{x^{2}}\right]_{\p_{1}}= k\left[\frac{y^{2}}{x}, \frac{x^{2}}{y^{3}}\right]_{\p_{2}},$$
 where $\p_{1}$ is the ideal generated by $x/y$ and $\p_{2}$ is the ideal generated by $y^{2}/x$.  
 
  \begin{prop} Let $\nu$ be the monomial valuation on $k(x,y)$ such that $\nu(x)=a$ and $\nu(y)=b$, where $a$ and $b$ are positive coprime integers such that $a > b$.  Then the valuation ring $R_{\nu}$ of $\nu$ is  $k[\frac{y^{a}}{x^{b}}, \frac{x^{p}}{y^{q}}]_{\p}$ where $\p=(\frac{x^{p}}{y^{q}})$ and $p, q\in \Z_{>0}$ such that $pa-qb=1$.
 \end{prop}
 
 \begin{proof}
First, we show that $k[\frac{y^{a}}{x^{b}}, \frac{x^{p}}{y^{q}}]_{\p}$ is contained in $R_{\nu}$.  Suppose that $r\in k[\frac{y^{a}}{x^{b}}, \frac{x^{p}}{y^{q}}]_{\p}$ meaning that $r=f/u$ for some $f\in k[\frac{y^{a}}{x^{b}}, \frac{x^{p}}{y^{q}}]$ and $u\not\in \p$. Let $f_{mn}\left(\frac{y^{a}}{x^b}\right)^m\left(\frac{x^p}{y^q}\right)^n$ be the monomial term in $f$ with smallest value. Therefore, since $\nu$ is the monomial valuation we know that 
\begin{equation*}\nu(r)=\nu(f/u)=\nu(f)-\nu(u)=\nu\left(f_{mn}\left(\frac{y^{a}}{x^{b}}\right)^{m}\left(\frac{x^{p}}{y^{q}}\right)^{n}\right)-\nu(u)=n-\nu(u).\end{equation*}
As $u$ is not in $\p$ we know that it is not divisible by $\frac{x^{p}}{y^{q}}$ meaning that one of its monomial terms is a polynomial in $\frac{y^{a}}{x^{b}}$, and hence $u$ has value zero since $\nu(\frac{y^{at}}{x^{bt}})=0$ for all $t\in \Z_{>0}$.  Thus, we see that the value of $r$ is equal to $n\geq0$ and so $r\in R_{\nu}$, implying the desired inclusion.

In order to see the other inclusion, let $f/g\in R_{\nu}$ where $f, g\in k[x,y]$ meaning that $\nu(f/g)\geq0$. Since $pa-qb=1$, it is the case that: 
$$\left(\frac{x^{p}}{y^{q}}\right)^{a}\left(\frac{y^{a}}{x^{b}}\right)^{q}=x^{pa-qb}=x \quad \quad \text{and} \quad \quad \left(\frac{x^p}{y^q}\right)^b\left(\frac{y^a}{x^{b}}\right)^p=y^{pa-qb}=y.$$
Thus $k[x,y]$ is a subring of $k[\frac{y^{a}}{x^{b}}, \frac{x^{p}}{y^{q}}]$ so we can view $f$ and $g$ as polynomials in $k[\frac{y^{a}}{x^{b}}, \frac{x^{p}}{y^{q}}]$.  Factoring out the maximal possible power of $x^{p}/y^{q}$ from $f$ and $g$ we can write $f$ and $g$ as
\begin{equation*}f=\left(\frac{x^{p}}{y^{q}}\right)^{n}h\quad \quad \text{and}\quad \quad g=\left(\frac{x^{p}}{y^{q}}\right)^{m}h'\end{equation*}
where $h$ and $h'$ are polynomials in $k[\frac{y^{a}}{x^{b}}, \frac{x^{p}}{y^{q}}]$ and  $x^{p}/y^{q}$ does not divide $h$ or $h'$ -- i.e. $h$ and $h'$ are not in $\p$.  Since $x^{p}/y^{q}$ does not divide $h$ or $h'$ there is at least one monomial term in $h$ and $h'$ of the form $\lambda (y^{a}/x^{b})^{k}$ for some $k\in \N$ meaning $\nu(h)=\nu(h')=0$ as $\nu(y^{a}/x^{b})=0$.  Therefore, we have that $\nu(f/g)=\nu((x^{p}/y^{q})^{n-m})=n-m$, and so as $\nu(f/g)$ is greater than or equal to zero, it is the case that $n\geq m$.  Thus, we can represent $f/g$ as 
\begin{equation*}\frac{f}{g}=\left(\frac{x^{p}}{y^{q}}\right)^{n-m}\frac{h}{h'},\end{equation*}
which precisely means that $f/g\in k[\frac{y^{a}}{x^{b}}, \frac{x^{p}}{y^{q}}]_{\p}$, and so $R_{\nu} = k[\frac{y^{a}}{x^{b}}, \frac{x^{p}}{y^{q}}]_{\p}$ where $\p$ is the principal ideal generated by $x^{p}/y^{q}$.
\end{proof}

Having dealt with the case when $\nu$ is an integral monomial valuation on $k(x,y)$ we can now generalize this  result to the case when $\nu$ is equivalent to an integral monomial valuation. If $\nu$ is a real monomial valuation on $k(x,y)$, then $\nu$ is clearly equivalent to a monomial valuation $\nu'$ where $\nu'(x)=\nu(x)/\nu(y)$ and $\nu'(y)=1$. From this we see that such a monomial valuation will be equivalent to an integral monomial valuation if and only if $\nu(x)/\nu(y)$ is a rational number. Since equivalent valuations have isomorphic valuation rings, we thus have the following generalization describing the valuation ring of a monomial valuation equivalent to an integral monomial valuation.

\begin{theorem}Let $\nu$ be a real valued monomial valuation on $k(x,y)$ such that $\nu(x)/\nu(y)=a/b$ is a positive rational number, where $a$ and $b$ are coprime integers greater than 1, and let $p, q\in \Z_{>0}$ such that $pa-qb=1$. Then the valuation ring $R_{\nu}$ of $\nu$ is  $k[\frac{y^{a}}{x^{b}}, \frac{x^{p}}{y^{q}}]_{\p}$ where $\p$ is the ideal generated by $(\frac{x^{p}}{y^{q}})$.\end{theorem}
 
 \begin{exm}
 Let $\nu$ be the monomial valuation on $k(x,y)$ defined by $\nu(x)=3\pi$ and $\nu(y)=2\pi$. While $\nu$ itself is not an integral monomial valuation, $\nu(x)/\nu(y)=3/2$ and so the above theorem tells us the valuation ring associated to $\nu$ is
 $$R_\nu=k\left[\frac{y^{3}}{x^{2}}, \frac{x}{y}\right]_{\p}$$
 where $\p$ is the ideal generated by $x/y$. Note this valuation is actually equivalent to the monomial valuation $\nu'$ defined by $\nu'(x)=3$ and $\nu'(y)=2$, which we consider at the beginning of this section.
\end{exm}

 %%%%%%%%%%%%%%
 
 \subsection{Relation to Continued Fractions}
 
We now have an explicit formulation of $R_{\nu}$  when $\nu(x)/\nu(y)$ is rational. In general, when $\nu(x)/\nu(y)$ is not rationally valued calculating $R_{\nu}$ is difficult. However, we can determine $P_{\nu}$ in general using the continued fraction expansion of $\nu(x)/\nu(y)$.  Utilizing this description of $P_{\nu}$, we calculate $R_{\nu}$ in general. 

First, we define the \textbf{monotone positive branch of type $(s,t)$}, denoted $B(s, t)$, to be the longest subpath in $P_{\nu}$ such that every vertex is of the form $k[s,t/s^{m}]$ for some $m \geq 1$. With this notation, $P_{\nu}$ decomposes as a disjoint union of positive branches. 

\begin{exm} For example, the monomial valuation $\nu$ on $k(x,y)$, where $\nu(x)=24$ and $\nu(y)=7$. The continued fraction expansion of $24/7$ is $[3;2,3]$ or $[3;2,2,1]$, and so $P_{\nu}$ is
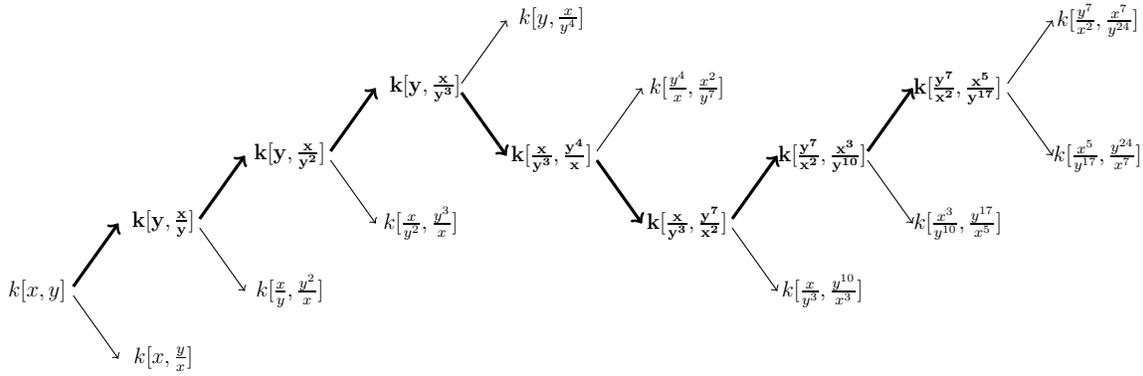
\begin{figure}[H]
\begin{center}
	\resizebox{!}{!}{\begin{tikzpicture}
	[scale=.6,colorstyle/.style={circle, draw=black!100,fill=black!100, thick,
	inner sep=0pt, minimum size=2 mm},every node/.style={scale=.7}]

	\node at (0,0){$k[x,y]$}; 
	
	\draw[very thick,->](.8,.1)--(1.8,1.5);
	\draw[->](.8,-.1)--(1.8,-1.5);
	\node at (2.8,1.5){$\mathbf{k[y, \frac{x}{y}]}$};
	\node at (2.8,-1.5){$k[x, \frac{y}{x}]$};
	
	\draw[very thick,->](3.6,1.6)--(4.6,3);
	\draw[->](3.6,1.4)--(4.6,0);
	
	\node at (5.6,3){$\mathbf{k[y, \frac{x}{y^{2}}]}$};
	\node at (5.6,0){$k[\frac{x}{y}, \frac{y^{2}}{x}]$};
	
	\draw[very thick,->](6.5,3.1)--(7.5,4.5);
	\draw[->](6.5,2.9)--(7.5,1.5);
	
	\node at (8.6,4.5){$\mathbf{k[y, \frac{x}{y^{3}}]}$};
	\node at (8.5,1.5){$k[\frac{x}{y^{2}}, \frac{y^{3}}{x}]$};
	
	\draw[->](9.4,4.6)--(10.4,6);
	\draw[very thick,->](9.4,4.4)--(10.4,3);
	
	\node at (11.4,6){$k[y, \frac{x}{y^{4}}]$};
	\node at (11.4,3){$\mathbf{k[\frac{x}{y^{3}}, \frac{y^{4}}{x}]}$};

	\draw[->](12.4,3.1)--(13.4,4.5);
	\draw[very thick,->](12.4,2.9)--(13.4,1.5);
	
	\node at (14.4, 4.5){$k[\frac{y^{4}}{x}, \frac{x^{2}}{y^{7}}]$};
	\node at (14.4, 1.5){$\mathbf{k[\frac{x}{y^{3}}, \frac{y^{7}}{x^{2}}]}$};
	\draw[very thick,->](15.4,1.6)--(16.4,3);
	\draw[->](15.4,1.4)--(16.4,0);
	
	\node at (17.4, 3){$\mathbf{k[\frac{y^{7}}{x^{2}}, \frac{x^{3}}{y^{10}}]}$};
	\node at (17.4,0){$k[\frac{x}{y^{3}}, \frac{y^{10}}{x^{3}}]$};

	\draw[very thick,->](18.4,3.1)--(19.4,4.5);
	\draw[->](18.4,2.9)--(19.4,1.5);
	\node at (20.4,4.5){$\mathbf{k[\frac{y^{7}}{x^{2}},\frac{x^{5}}{y^{17}}]}$};
	\node at (20.4,1.5){$k[\frac{x^{3}}{y^{10}}, \frac{y^{17}}{x^{5}}]$};
	
	\draw[->](21.5,4.6)--(22.5,6);
	\draw[->](21.5,4.4)--(22.5,3);
	
	\node at (23.5,6){$k[\frac{y^{7}}{x^{2}},\frac{x^{7}}{y^{24}}]$};
	\node at (23.5,3){$k[\frac{x^{5}}{y^{17}}, \frac{y^{24}}{x^{7}}]$};
	\end{tikzpicture}}
	\end{center}
	\caption{$P_{\nu}$ for $\nu(x)=24, \nu(y)=7$}
\end{figure}
In the above notation,
$$P_{\nu}=B(y,x)\cup B(x/y^3,y)\cup B(y^7/x^2,x/y^3).$$
Further, $B(y,x)$ has length three while $B(x/y^3,y)$ and $B(y^7/x^2,x/y^3)$ have length two. 
%The lengths of these paths correspond to the first three digits in the continued fraction expansion of $24/7: [3;2;2;1]$. 
\end{exm}

Since our assumptions on $\nu(x)$ and $\nu(y)$ ensure that $k[x,y]$ and $k[y,x/y]$ will always be the first two vertices in $P_{\nu}$, we have that
\begin{equation*}P_{\nu}=B\left(y,x\right)\cup B\left(\frac{x}{y^{m_{1}}}, y\right)\cup B\left(\frac{y^{m_{2}m_{1}+1}}{x},\frac{x}{y^{m_{1}}} \right)\cup\cdots\cup B(s,t)\cup B\left(\frac{t}{s^{m_{i}}}, s\right)\cup\cdots\end{equation*}
for some $m_i\in \N$. Note in a slight abuse of notation we consider $k[x,y]$ to be in $B(y,x)$ (Since $k[x,y]=k[y,x/y^0]$).  So in essence to understand $P_{\nu}$, we must only know the $m_{i}$'s.  Since the first vertex in $B(x/y^{m_{1}}, y)$ directly follows the last vertex of $B(y, x)$, which by definition has the form $k[y, x/y^{t}]$ for some $t$, we have that $m_{1}$ is the length of $B(y,x)$.  By a similar argument, we have $m_{2}=\length(B(x/y^{m_{1}}, y))+1$ and in general that $m_{i}=\length(B(s,t))+1$ when $i\neq1$.  Note the case when $i=1$ is special because of the abuse of notation mentioned above.
 
Since $B(y,x)$ is the longest positive path with vertices of the form $k[y, x/y^{m}]$ and $\nu(x/y^{m})=\nu(x)-m\nu(y)$, it is the case that $k[y, x/y^{m}]$ will be positive for $1\leq m \leq \lfloor\nu(x)/\nu(y) \rfloor$ assuming $\nu(y)\neq1$. If, however, $\nu(y)=1$, then this inequality becomes $1\leq m < \lfloor\nu(x)/\nu(y) \rfloor$.  Therefore, recalling Section \ref{section: CF}, we find that $\lfloor \nu(x)/\nu(y)\rfloor$ is equal to $d_{0}$ in the continued fraction expansion of $\nu(x)/\nu(y)$, and thus $m_{1}=d_{0}$.  Applying this same argument to the general case we get that $m_{i}=d_{i-1}$. 

When $\nu(x)/\nu(y)$ is rational, $P_{\nu}$ terminates, as does the continued fraction expansion of $\nu(x)/\nu(y)$. However, when $\nu(x)/\nu(y) \in \Q - \Z$, it has two viable expansions, namely $$\nu(x)/\nu(y) = [d_0;d_1,\ldots,d_n]=[d_0;d_1,\ldots,d_{n}-1,1].$$ To find $P_{\nu}$ in this case, we consider the longer continued fraction expansion up to the last term. So the $m_i$'s for $\nu$ are $d_0,d_1,\ldots,d_{n}-1$.

\begin{exm}Once again returning to Example \ref{Val23} where $\nu$ is the monomial valuation on $k(x,y)$ such that $\nu(x)=3$ and $\nu(y)=2$.  From our previous work we know that $P_{\nu}$ is the path given by
\begin{figure}[H]
\begin{center}
 	\begin{tikzpicture}
	[scale=.75,colorstyle/.style={circle, draw=black!100,fill=black!100, thick,
	inner sep=0pt, minimum size=2 mm}]
	\node at (0,0){$k[x,y]$};
	\node at (5,0){$k[y, x/y]$};
	%\node[gray] at (5,-3){$\C[x, y/x]$};
	
	%\node[gray] at (10, 4.5){$\C[y, x/y^{2}]$};
	\node at (10.8, 0){$k[x/y, y^{2}/x]$};
	%\node at (10, -1.5){$\C[x, y/x^{2}]$};
	%\node at (10, -4.5){$\C[y/x, x^{2}/y]$};
	
	%\node[gray] at (16, 3){$\C[x/y, y^3/x^2]$};
	%\node[gray] at (16, 0){$\C[y^2/x, x^2/y^3]$};
	
	\draw[->,ultra thick](.8,0)--(3.8,0);
	%\draw[->,thick,gray](.8,-.1)--(3.8,-3);
	%\draw[thick](6.1,-2.9)--(8.65, -1.5);
	\draw[->,ultra thick](6.1,0)--(9.1, 0);
	%\draw[->,thick, gray](6.1,3.1)--(8.65,4.5);
	%\draw[thick](6.1,-3.1)--(8.4,-4.5);
	
	%\draw[->,thick,gray](11.5,1.6)--(14.3,3);
	%\draw[->,thick,gray](11.5,1.4)--(14.2,0);
	
	%\draw[dashed](11.2,1.5)--(12.5,2);
	%\draw[dashed](11.2,1.5)--(12.5,1);
	
	%\draw[dashed](11.2,-1.5)--(12.5,-1);
	%\draw[dashed](11.2,-1.5)--(12.5,-2);
	%\draw[dashed](11.5,-4.5)--(12.5,-5);
	%\draw[dashed](11.5,-4.5)--(12.48,-4);
	
	\end{tikzpicture}
\end{center}
 \caption{$P_{\nu}$ for monomial valuation given by $\nu(x)=3$ and $\nu(y)=2$.}
\end{figure}

So both $B(y,x)$ and $B(x/y, y)$ have length one, corresponding to the first two digits of the continued fraction $3/2 = [1;1,1]$ from Example \ref{ExmConFrac}.  
\end{exm}
 
So given a monomial valuation $\nu$ on $k(x,y)$ the path $P_{\nu}$ is determined by the continued fraction expansion of $\nu(x)/\nu(y)$.  Since as we showed in Section \ref{section: CF} a real number $r$ has a finite continued fraction if and only if $r$ is rational, we now know when  $P_{\nu}$ is finite.

\begin{prop}Let $\nu$ be a monomial valuation on $k(x,y)$. Then $P_{\nu}$ is finite if and only if $\nu(x)/\nu(y)$ is rational.\end{prop}

Thus, if $\nu(x)/\nu(y)$ is irrational, then $P_{\nu}$ is infinite, and so there is no terminal object.  However, $P_{\nu}$ can still be used to describe the valuation ring associated to $\nu$.

\begin{theorem}\label{LastTheorem} If $\nu$ is a monomial valuation on $k(x,y)$ such that $\nu(x)/\nu(y)$ is irrational, then 
$$ R_{\nu}=\bigcup_{R_{i}\in V(P_{\nu})}R_{i_{\m_{i}}}$$
where $V(P_{\nu})$ is the vertex set of $P_{\nu}$ and $\m_{i}=\m_{\nu}\cap R_{i}$, with $\m_{\nu}$ the maximal ideal of $R_{\nu}$.
\end{theorem}

\begin{proof}First, let us show that $R_{i_{\m_{i}}}$ is contained in $R_{\nu}$.  Towards this goal, let $r\in R_{i_{\m_{i}}}$, meaning that $r=f/u$ where $f\in R_{i}$ and $u\not\in \m_{i}$.  By the definition of $P_{\nu}$, every element of $R_{i}$ has non-negative values so $\nu(f)\geq0$.  Further, since $\m_{i}=R_{i}\cap \m_{\nu}$, we know $\nu(u)=0$.  Thus, we have that $\nu(r)=\nu(f)-\nu(u)\geq0$, meaning 
$$\bigcup_{R_{i}\in V(P_{\nu})}R_{i_{\m_{i}}}\subset R_{\nu}.$$
The other inclusion follows from Exercise 4.12(3) of Chapter II of \cite{hartshorne_algebraic_1977}. \end{proof}
\subsection{General Monomial Valuations}

In summary, any $\R$-valued monomial valuation $\nu$ determines a unique path within the valuation tree $T$.  Moreover, these paths can be used to find the valuation ring associated with $\nu$.  Conversely, any finite path (and certain infinite paths) within $T$ will correspond to a unique monomial valuation.

While up until this point we have only considered $\R$-valued valuations, much of what we have done holds for any totally ordered abelian group $\Gamma$. In particular, if $\nu$ is a $\Gamma$-valued monomial valuation on $k(x,y)$, then we can still associate $\nu$ with a path within $T$ as we described previously. Additionally, the description of the valuation ring given in Theorem \ref{LastTheorem} holds regardless of the choice of $\Gamma$. See Exercise 5.6 of Chapter V of \cite{hartshorne_algebraic_1977}.

As mentioned above not every path in the tree $T$ corresponds to a $\R$-valued monomial valuation on $k(x,y)$. In particular, infinite paths where all but finitely many vertices are of the form $k[f, g/f^t]$ for some Laurent monomial $f$ and $g$ do not correspond to $\R$-valued monomial valuations. However, these paths do correspond to monomial valuations on $k(x,y)$ with value group other than $\R$. For example, consider the following: 
%The only paths in $T$ which do not correspond to a $\R$-valued valuations are paths where all but finitely many vertices are of the from $k[f, g/f^t]$. However, if we allow for a different choice of totally ordered abelian group, the correspondence between paths in $T$ and equivalence classes of monomial valuations is bijective. For example, consider the following:

\begin{exm}

Let $P$ be a path in $T$, which does not correspond to a $\R$-valued monomial valuation. Then all but finitely many of the vertices of $P$ are of the form $k[f, g/f^t]$ for $t\in \N$ where $f$ and $g$ are Laurent monomials. Our goal is to construct a $\Z^2$-valued monomial valuation on $k(x,y)$ corresponding to $P$ where $\Z^2$ has lexicographic order. Towards this, note that if $\nu$ is a $\Z^2$-valued monomial valuation on $k(x,y)$ such that $\nu(g)=(1,0)$ and $\nu(f)=(0,1)$ the $\nu(g/f^t)=(1,-t)$. As $\Z^2$ has lexicographic ordering, this would mean $g/f^t$ is $\nu$-positive for all $t$ and all vertices of the form  $k[f, g/f^t]$ are in the path $P_\nu$. 

Thus, to find a $\Z^2$-valued monomial valuation $\nu$ on $k(x, y)$ corresponding to $P$, it suffices to show that we can construct such a valuation where $\nu(f)=(0,1)$ and $\nu(g)=(1,0)$. However, we know that $k[f,g]$ contains $k[x,y]$, so we can write both $x$ and $y$ as polynomials in terms of $f$ and $g$. Thus, once we specify the values of both $f$ and $g$, we can solve for the values of $x$ and $y$ afterwards. So to construct a $\Z^2$-valued monomial valuation $\nu$ corresponding to $P$, simply let $\nu(f)=(0,1)$ and $\nu(g)=(1,0)$. 

\end{exm}

\section{Resolution of Cusp Singularities}

We now shift our attention to examining how certain monomial valuations are related to a resolution process for plane cusps.  Specifically, we describe a resolution for the plane curve $x^b=y^a$, where $a, b \in \Z_{> 1}$ are relatively prime integers, over an arbitrary field $k$.  To do this, we show an explicit connection between the process of resolving this singularity and the process of finding the path $P_{\nu}$ as described in Section \ref{Val-Tree}.  Finally, we show how resolving the singularities of this family of curves is related to continued fractions. 

For $p \in \A^2$, let $\pi:B_p(\A^2)\arrow{}\A^2$ be the blowing up of $\A^2$ at $p$. Recall that $\pi$ is an isomorphism except over $p$, and the fiber $E$ over $p$ is called the exceptional set. Further $B_p(\A^2)$ is covered by two charts $U_1$ and $U_2$, both isomorphic to $\A^2$. Choosing coordinates $x,y$ for $\A^2$ so that $p=(0,0)$, then $U_1$ has coordinates $x, y/x$ and $U_2$ has coordinates $y, x/y$. The exceptional set $E$ is defined by $x$ in $U_1$ and by $y$ in $U_2$. 

Now if we consider a curve $C$ in $\A^2$, which passes through $p$ and is defined by some polynomial $f(x,y)$, then the preimage of $C$ under $\pi$ is defined by $f\circ \pi$. Moreover, $\pi^{-1}(C)$ decomposes as the union of two curves 
$$\pi^{-1}(C)=\tilde{C}\cup E$$
 where $E$ is the exceptional set, $\pi^{-1}(\{p\})$, and $\tilde{C}$ is the closure of $\pi^{-1}(C\setminus \{p\})$ in $B_p(\A^2)$, which is called the proper transform of $C$. In the chart $U_1$, for example, the preimage of $X$ is defined by 
$$f\left(x, x\left(\frac{y}{x}\right)\right)=x^d\tilde{f}\left(x, \frac{y}{x}\right)$$
where $d$ is the degree of the monomial occurring in $f$ of lowest degree. Note that the factor $x^d$ defines the exceptional set $E$, (``counted $d$-times'') and the factor $\tilde{f}$ defines $\tilde{C}$. 

With this set-up, we can now define what we mean by a resolution of singularities of a curve.

\begin{defn}A curve $C$, which is a union of irreducible components $C=C_1\cup\cdots\cup C_n$, on a surface is resolved if
\begin{enumerate}
\item each $C_i$ is smooth.
\item no three components meet at a point.
\item no two curves are tangent at any point, that is, if $C_i$ and $C_j$ meet at a point then their tangent lines at that point are different.
\end{enumerate}
\end{defn}
If a curve $C$ satisfies (2) and (3) above, then $C$ is said to have normal crossings. Thus, an alternative definition for a curve to be resolved is that all of its components are smooth and meet only in normal crossings. Further, a resolution of singularities of a curve $C$ is a sequence of blow-ups such that the preimage of $C$ under these blow ups is resolved. 

Before describing our results in general, we find it useful to first consider an example of a resolution of singularities that both illustrates the above definitions and highlights our results to come. 

\begin{exm}\label{Exm: Rez 1}
Let us look at the affine variety $V=\V(x^{2}-y^{3})$, i.e. the solution set of $x^2=y^3$, over an arbitrary field $k$. The Jacobian matrix associated to $V$ is $J=[2x, 3y^{2}]$ and so $J_{p}$ has rank zero only at $p=(0,0)$ meaning that $V$ is singular only at the origin. As considered above, let $\pi:B_{p}(\A^2)\arrow{}\A^2$ be the blowing up map of $\A^2$ at $p=(0,0)$. As described above, $B_p(\A^2)$ is covered by two charts $U_1$ and $U_2$, both isomorphic to $\A^2$, with coordinates $x, y/x$ and $y, x/y$ respectively. When we look at $\pi^{-1}(V)$ in the $U_2$ chart, a simple calculation shows that:
\begin{equation*}\pi^{-1}(V)\cap U_2=\V\left(y^{2}\right)\cup\V\left(\left(\frac{x}{y}\right)^{2}-y\right).\end{equation*}
\begin{figure}[H]\label{Hyperbola}
\begin{tikzpicture}
 \draw[ultra thick, domain=0:2,smooth,variable=\t] plot ({\t},{\t*\t});
 t\draw[ultra thick, domain=0:-2,smooth,variable=\t] plot ({\t},{\t*\t});
 \draw [thick, <->] (0,4) -- (0,0) -- (4,0);
 \draw [thick, <->] (0,-2) -- (0,0) -- (-4,0);
 \node at (0,4.3) {$y$};
 \node at (4.3,0) {$\frac{x}{y}$} ;
% \node[draw,circle,inner sep=1.5pt,fill] at (1,1) {};
 \draw [ultra thick, <->] (-4,0)--(4,0);
 
 \node[scale=.5] at (2.6,2.1) {$\tilde{V}=\V\left(y=\left(\frac{x}{y}\right)^2\right)$};
  \node[scale=.6] at (-2.6,-0.6) {$E=\V(y^2)$};
\end{tikzpicture}
\caption{Picutre of $\pi^{-1}(V)$ in the $U_2$ chart with coordinates $y, x/y$.}
\end{figure}
Note that $\V(y^2)$ corresponds to the exceptional set in $U_2$ while the other factor is the proper transform of $V$ in $U_2$. In fact, $\pi^{-1}(V)\cap U_2$ satisfies conditions (1) and (2) for being resolved -- i.e. irreducible components are smooth and no more than two components intersect. However, $\pi^{-1}(V)\cap U_2$ has a non-normal crossing. More precisely, in this chart the exceptional set is tangent to the proper transform at $(0,0)$. So $\pi^{-1}(V)$ is not resolved. 

Turning our attention to the $U_1$ chart of $B_p(\A^2)$ we see that in this chart $\pi^{-1}(V)$ is given by\begin{equation*}\pi^{-1}(V)\cap U_1=\V\left(x^{2}\right)\cup\V\left(1-\left(\frac{y}{x}\right)^{3}x\right).\end{equation*}
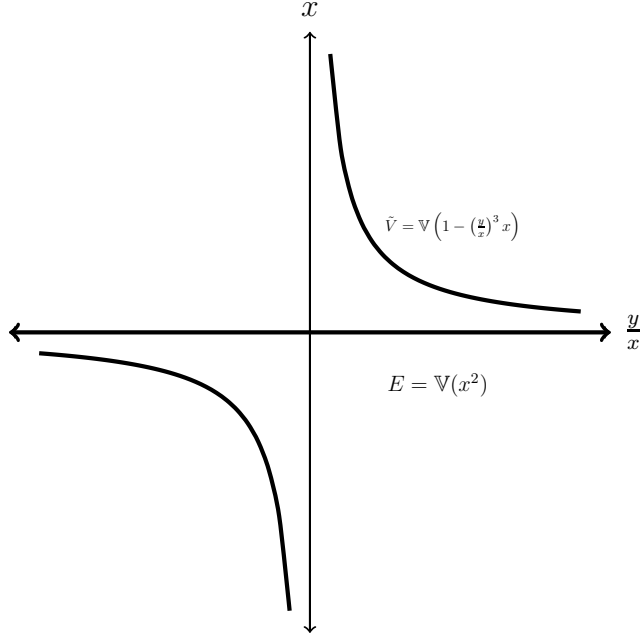
\begin{figure}[H]\label{Hyperbola}
\begin{tikzpicture}
 \draw[ultra thick, domain=.27:3.6,smooth,variable=\t] plot ({\t},{1/\t});
 \draw[ultra thick, domain=-.27:-3.6,smooth,variable=\t] plot ({\t},{1/\t});
 \draw [thick, <->] (0,4) -- (0,0) -- (4,0);
 \draw [thick, <->] (0,-4) -- (0,0) -- (-4,0);
 \draw [ultra thick, <->] (4,0)--(-4,0);
 \node at (0,4.3) {$x$};
 \node at (4.3,0) {$\frac{y}{x}$} ;
% \node[draw,circle,inner sep=1.5pt,fill] at (1,1) {};
 \node[scale=.5] at (1.9,1.4) {$\tilde{V}=\V\left(1-\left(\frac{y}{x}\right)^{3}x\right)$};
 \node[scale=.7] at (1.7,-.7) {$E=\V(x^2)$}; 
\end{tikzpicture}
\caption{Picture of $\pi^{-1}(V)$ in the $U_1$ chart with coordinates $x, y/x$.}
\end{figure}
Similar to what we saw in the $U_2$ chart, $\V(x^2)$ is the exceptional set in $U_1$ while the other factor is the proper transform of $\pi^{-1}(V)$ in $U_1$. However, unlike in the other chart $\pi^{-1}(V)\cap U_1$ is smooth and has only normal crossings, and so is resolved. 

%We will call $\V(y^2), \V(x^2)$, and anything of the form $\V(f^Bg^A)$ (where $f, g \in k(x,y)$) the ``exceptional part" of the chart. Anything not of this form (i.e. $\V(f^{\beta}-g^{\alpha}), f, g \in k(x,y), \alpha, \beta \in \Z_{\geq0}$) is called the``non-exceptional part" of the chart.

%From this we see that in each chart $B_{p}(V)$ is nonsingular, however, in the $U_1$ chart $B_{p}(V)$ has a non-normal crossing, since the non-exceptional part is tangent to the exceptional part at $(0,0)$.  Recall that two affine varieties $V_{1}, V_{2}\subset \A^{n}$ intersect normally at a $p\in \A^{n}$ if and only if 
%\begin{equation*} \mathrm{codim} T_{p}(V_{1}\cap V_{2})=\mathrm{codim} T_{p}(V_{1})+\mathrm{codim} T_{p}(V_{2}).\end{equation*}
%As a curve is resolved only when it has no singularities and no normal crossings, 

Since $\pi^{-1}(V)\cap U_2$ is not resolved we must blow up $U_2$ at its origin in order to try and remove the tangential intersection. This blow-up has two charts, $U_3$ and $U_4$ both isomorphic to $\A^2$, with coordinates $x/y, y^2/x$ and $y, x/y^2$, respectively.  Looking in the $U_3$ chart, we see that the preimage of $\pi^{-1}(V)\cap U_2$ is given by:
\begin{equation*}  \V\left(\left(\frac{x}{y}\right)^3\left(\frac{y^2}{x}\right)^2\right)\cup \V\left(\frac{x}{y}-\frac{y^2}{x}\right).
\end{equation*}
In the $U_4$ chart, the preimage is given by: 
\begin{equation*} \V (y^3) \cup \V \left(1-\left(\frac{x}{y^2}\right)^2y\right).
\end{equation*}
In the $U_4$ chart, the blow up of the curve is smooth and has normal crossings. In the $U_3$ chart, it is smooth, but consists of three lines intersecting at a single point (the origin), which is a non-normal crossing. So we must once again blow up at the origin, this time in the $U_3$ chart. Again, this blow up has two charts $U_5, U_6$, both isomorphic to $\A^2$, with coordinates $x/y, y^3/x^2$ and $y^2/x , x^2 /y^3$, respectively. In these charts, our curve of interest is given by, respectively, 
\begin{equation*} \V\left( \left(\frac{x}{y}\right)^6\left(\frac{y^3}{x^2}\right)^2 \right) \cup \V\left( 1- \frac{y^3}{x^2} \right), \mbox{ and } 
\V \left( \left(\frac{x^2}{y^3}\right)^3\left(\frac{y^2}{x}\right)^6\right) \cup \V \left( \frac{x^2}{y^3}-1\right).
\end{equation*}
In both charts, the curve is smooth with normal crossings. Thus we have resolved $V=\V(x^{2}-y^{3})$, and the sequence of blow ups we preformed is a resolution of singularities for $\V(x^2-y^3)$. 

Looking at this process we see that in every step, each blow up is covered by two charts. If we take the coordinate rings associated with these charts and form a graph where the edges are induced by inclusions, we obtain the tree:
\begin{figure}[H]
\begin{center}
 	\begin{tikzpicture}
	[scale=.75,colorstyle/.style={circle, draw=black!100,fill=black!100, thick,
	inner sep=0pt, minimum size=2 mm}]
	\node at (0,0){$\mathbf{k[x,y]}$};
	\node at (5,3){$\mathbf{k[y, x/y]}$};
	\node at (5,-3){$k[x, y/x]$};
	
	\node at (10, 4.5){$k[y, x/y^{2}]$};
	\node at (10, 1.5){$\mathbf{k[x/y, y^{2}/x]}$};
	%\node at (10, -1.5){$\C[x, y/x^{2}]$};
	%\node at (10, -4.5){$\C[y/x, x^{2}/y]$};
	
	\node at (16, 3){$k[y^2/x, x^2/y^3]$};
	\node at (16, 0){$k[x/y, y^3/x^2]$};
	
	\draw[->,ultra thick](.8,.1)--(3.8,3);
	\draw[->,thick](.8,-.1)--(3.8,-3);
	%\draw[thick](6.1,-2.9)--(8.65, -1.5);
	\draw[->,ultra thick](6.1,2.9)--(8.4, 1.5);
	\draw[->, thick](6.1,3.1)--(8.65,4.5);
	%\draw[thick](6.1,-3.1)--(8.4,-4.5);
	
	\draw[->,thick](11.5,1.6)--(14.3,3);
	\draw[->,thick](11.5,1.4)--(14.2,0);
	
	%\draw[dashed](11.2,1.5)--(12.5,2);
	%\draw[dashed](11.2,1.5)--(12.5,1);
	
	%\draw[dashed](11.2,-1.5)--(12.5,-1);
	%\draw[dashed](11.2,-1.5)--(12.5,-2);
	%\draw[dashed](11.5,-4.5)--(12.5,-5);
	%\draw[dashed](11.5,-4.5)--(12.48,-4);
	
	\end{tikzpicture}
\end{center}
 \caption{Tree resulting from the resolution of $\V(x^{2}-y^{3})$.}
\end{figure}
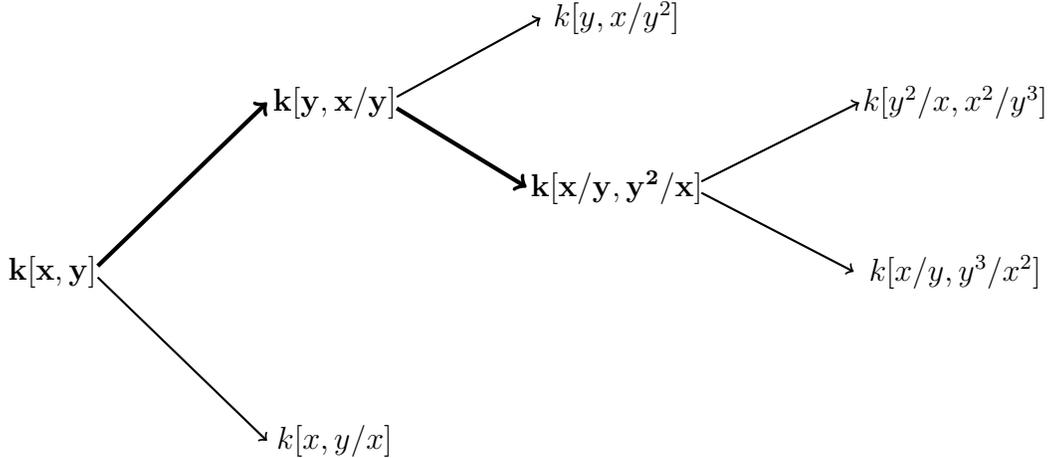
Moreover, if we consider the subgraph whose vertices are the coordinate rings corresponding to charts in which there was either a singularity or a non-normal crossing, we get a path. In fact, the path we get is identical to the path we constructed in Examples 3.8 and 3.12 for the $\R$-valued monomial valuation $\nu$ on $k(x,y)$ given by $\nu(x)=3$ and $\nu(y)=2$. Since the path $P_\nu$ is determined by the continued fraction of $\nu(x)/\nu(y)=3/2$ we see that in this case the continued fraction determines a resolution of singularities by telling us which charts feature either singularities or non-normal crossings. 

% we have that the continued fraction expansion of $3/2$ tells us which chart will contain a singularity or non-normal crossing. 

\end{exm}

Of course, nothing is special about the choice of two and three, and for any positive relatively prime integers $a$ and $b$ both larger than 1, we have a similar relationship.  However, before stating the general result, we must formalize the notions of the tree and path discussed in the above example.

The coordinate tree $T$ is a infinite directed tree whose vertices are certain subrings of $k(x,y)$.  The ring $k[x,y]$ is a vertex of $T$, and then we define the vertex set of $T$ recursively.  In particular, if $k[f,g]$ is a vertex of $T$ then $k[f, g/f]$ and $k[g, f/g]$ are also vertices of $T$, and there are edges going from $k[f,g]$ to $k[f, g/f]$ and $k[g, f/g]$.  Put another way, $T$ is a tree whose vertices are the coordinate rings associated to the charts of repeated blow-ups of the origins in each of the newly-created charts isomorphic to $\A^2$.  Note: from this construction we obtain a tree identical to the valuation tree defined in Section \ref{Val-Tree}.

Given positive relatively prime integers $a$ and $b$, we want to associate a subgraph of $T$ with the curve $V=\V(x^{b}-y^{a})$. The idea is that this subgraph will tell us exactly what points need to be blown up to resolve the curve $V$. Define $P_{a,b}$ to be the subgraph of $T$ such that each vertex in $P_{a,b}$ is a coordinate ring corresponding to an affine chart in which $V$ has either a singularity or a non-normal crossing. 
 
%If we consider the tree of coordinate rings representing the charts of all possible blow-ups of curves of the form $\V(x^b-y^a)  \forall  a,b \in \Z_{>0}$ we obtain an infinite tree identical to $T$ from section INTERNAL REFERENCE NEEDED.  In the resolution process each blow up contained precisely one smooth chart without non-normal crossing.  Thus, the coordinate rings associated to the charts in which there are singularities and non-normal crossings corresponds to a path in $T$.  Moreover this path is exactly the same path, $P_{\nu}$, we had in INTERNAL REFERENCE NEEDEDED where $\nu(x)=a, \nu(y)=b$.

One hopes that as in the example above $P_{a,b}$ is a path, and moreover it is identical to $P_{\nu}$ defined in \ref{Val-Tree}.  In sum, we want to prove the following theorem relating the resolution of the curve $V$ to the process of finding the valuation rings of monomial valuations described in \ref{Val-Tree}.  

\begin{theorem}\label{Thm: resolution rod}Let $a,b\in \Z_{>1}$ such that $(a,b)=1$ and $a>b$.  Then $P_{a,b}=P_{\nu}$, where $\nu$ is the monomial valuation on $k(x,y)$ given by $\nu(x)=a$ and $\nu(y)=b$.\end{theorem}

This theorem allows us to explicitly resolve the singularities of the curve $x^b=y^b$ by blowing up a sequence of points prescribed by $P_{a,b}$. See Remark 4.5. Before we prove the above theorem, let us establish the following crucial lemma.

 \begin{lem}\label{lem: path lemma}Let $a,b\in \Z_{>1}$ such that $(a,b)=1$ and $a>b$. Then $P_{a,b}$ is a path.
\begin{proof} By the definition of $P_{a,b}$, this means we must precisely show that the blowup corresponding to a given vertex of $P_{a, b}$ will feature at most one chart in which the preimage curve is either singular or has a non-normal crossing. To prove this, we shall proceed by induction. For our base case, we consider $k[x,y]$, which is in $P_{a,b}$. If we blow up $\A^2_{(x,y)}$ at the origin, we obtain the curves $\V(y^b)\cup\V((x/y)^b-y^{a-b})$ and $\V(x^b)\cup\V(1-(y/x)^ax^{a-b})$, which correspond to the coordinate rings $k[y,x/y]$ and $k[x, y/x]$, respectively. Since we took $a$ and $b$ to be relatively prime, the curve is smooth in the second chart, and is either singular or has a non-normal crossing only at the origin in the first chart, so the second node in $P_{a,b}$ is $k[y,x/y]$.

Let $k[f,g]$ be a coordinate ring in $P_{a,b}$. Our claim is that at most one of the two subsequent rings in $T_{a,b}$ is associated with a chart in which the preimage curve features a singularity or non-normal crossing. By induction assume that in the chart associated with $k[f,g]$, a chart featuring a singularity or non-normal crossing, the original curve (or its preimage) has irreducible decomposition of the form (*) $\V(f^Ag^B)\cup\V(f^{\beta}-g^{\alpha})$ for integers $A, B \ge 0$, and coprime positive integers $\alpha , \beta$; without loss of generality, $\alpha \ge \beta \ge 1$. If $\alpha = \beta$, then in fact $\alpha = \beta = 1$,\footnote{For we can factor $f^n - g^n = (f-g) \sum_{i=0}^{n-1} f^{n-1-i} g^i $ ($n > 1$), in which case the principal ideal it generates in the polynomial ring $k[f, g]$ is not prime, whence $\V(f^n-g^n)$ is not irreducible.} and so the chart features three lines meeting only at the origin, which is a non-normal crossing; so proceeding as in Example \ref{Exm: Rez 1}, we know that blowing up this chart at the origin will yield two new charts in which the respective preimage curves are each smooth with normal crossings. So without loss of generality, assume $\alpha > \beta > 1$. If we perform another blowup and look at the two charts, which have coordinates $(g, \frac{f}{g})$ and $(f, \frac{g}{f})$, respectively, then the preimage curves have respective irreducible decompositions
$$\V\left(g^{B+\beta+A}\left(\frac{f}{g}\right)^A\right) \cup \V\left(\left(\frac{f}{g}\right)^{\beta}-g^{\alpha-\beta}\right) \mbox{ and }\V\left(\left(\frac{g}{f}\right)^Bf^{\beta+A+B}\right) \cup \V \left(1-\left(\frac{g}{f}\right)^{\alpha}f^{\alpha-\beta} \right).$$
 Notice that in the second chart, the preimage curve is now smooth and has no non-normal crossings. So we consider what happens in the $k[g, \frac{f}{g}]$ chart, which of course features an irreducible decomposition of form (*).

First, suppose $\alpha-\beta\neq1$. Then the first preimage curve will have a cusp at the origin, and will therefore have a singularity. Now suppose $\alpha-\beta=1$. Then $\V\left(\left(\frac{f}{g}\right)^{\beta}-g^{\alpha-\beta}\right)$ is smooth. If we have that $A \neq 0$, then the first chart will feature a non-normal crossing, namely three curves meeting at the origin. 

If instead $A=0$, then in the chart $\A^2_{(f, g)}$, the curve has irreducible decomposition $\V(f^Ag^B)\cup\V(f^{\beta}-g^{\alpha}) = \V(g^B)\cup\V(f^{\beta}-g^{\alpha})$. When we multiply together the equations defining the irreducible components of the preimage curve, in any chart at any stage in this blowup process, we must obtain the polynomial equation $x^b-y^a=0$ up to a sign change. Therefore, without loss of generality say $f^{\beta}g^B-g^{\alpha+B}=x^b-y^a$. Since $x^b-y^a$ is irreducible, $B = 0$, so that $g=y$ and $f=x$, and so we must be considering the first blowup, which we have previously established has exactly one chart containing a singularity or non-normal crossing. The induction is complete, so we conclude that $P_{a,b}$ is a path in $T$. 
\end{proof}
 \end{lem}
 
 Now that we know that $P_{a,b}$ is in fact a path we can prove our main result.

 \begin{proof}[Proof of Theorem \ref{Thm: resolution rod}]
 We will prove by induction that $P_{\nu}=P_{a,b}$; by the lemma, it suffices to prove they have the same vertices in succession. Also, we will show that if a node $k[f,g]$ is in $P_{\nu}$ (and therefore $P_{a,b}$), then the associated curve at this step in the resolution process is of the form $\V(f^Ag^B)\cup \V(f^{\nu(g)}-g^{\nu(f)})$ with $A, B$ some nonnegative integers, and with $( \nu(f) , \nu(g)) = 1$ (so we can invoke Euclid's algorithm, later).
 
  Let $a,b \in \Z_{>0}$ such that $a>b$ and $(a,b)=1$, and $\nu$ be the monomial valuation on $k(x,y)$ such that $\nu(x)=a, \nu(y)=b$. The first node in $T$ (i.e. the smallest ring in $T$) is $k[x,y]$, and the curve in this chart is given by $\V(x^b-y^a)=\V(x^{\nu(y)}-y^{\nu(x)})$. Clearly $( \nu(x) , \nu(y)) = 1$.
 
 The second node in $P_{\nu}$ is $k[y, x/y]$, with $\nu(y)=b, \nu(x/y)=a-b$. Likewise, the second node in $P_{a,b}$ is $k[y,x/y]$, and the curve in this chart is given by $\V(y^b)\cup \V((x/y)^b-y^{a-b})=\V(y^b) \cup \V((x/y)^{\nu(y)}-y^{\nu(x/y)})$. Clearly $(\nu(y) , \nu(x/y)) = (b, a-b) = (b, a) = 1$.
 
 So we see that the first two vertices of $P_{\nu}$ and $P_{a,b}$ agree, and the desired relationship holds.
 
 Now suppose that this relationship holds and $P_{\nu}$ is equal to $P_{a,b}$ up until some step where the vertex in both paths is $k[f,g]$, and the preimage curve for this chart is $\V(f^Ag^B)\cup \V(f^{\beta}-g^{\alpha})$, where $\nu(f)=\alpha$ and $\nu(g)=\beta$ are coprime positive integers. First suppose $\alpha \neq \beta$; without loss of generality, assume $\alpha$ is greater than $\beta$. So the next ring in $P_{\nu}$ is $k[g,f/g]$. Since $k[f,g] \in P_{a,b}$, the curve in this chart either has a singularity or non-normal crossing. Blowing up again, the proof of the previous lemma the coordinate ring shows that $k[g, f/g]$ lies in $P_{a,b}$, and the corresponding preimage curve has irreducible decomposition $$\V(f^Ag^{B+\beta}) \cup \V((f/g)^{\beta}-g^{\alpha-\beta})=\V(f^Ag^{B+\beta}) \cup \V((f/g)^{\nu(g)}-g^{\nu(f/g)}).$$ Further, $(\nu(g), \nu(f/g)) = (\nu(g), \nu(f)-\nu(g)) = (\nu(g), \nu(f)) = 1$. Thus by induction, we know that $P_{\nu}=P_{a,b}$. 
 
The last node of $P_{\nu}$ has form $k[f,g]$, with $\nu(f)=\nu(g)= 1$. (Since the nodes in $T$ following the final node in $P_{\nu}$ are $k[f,g/f], k[g,f/g]$, where $\nu(f/g)=\nu(g/f)=0$, so $\nu(g)=\nu(f)$; both equal 1 by Euclid's algorithm.)  The associated curve in this chart has irreducible decomposition $\V(f^Ag^B)\cup\V(f-g)$, that is, three lines meeting only at the origin. Blowing up this chart, we obtain the following preimage curves in the charts corresponding to $k[f,g/f]$ and $k[g,f/g]$, respectively: $\V(f^{A+1}g^B)\cup\V(1-(g/f))$ and $\V(f^Ag^{B+1})\cup \V((f/g)-1)$. Each curve is smooth with normal crossings. Thus the curve is completely resolved if we blow up at the origin in the chart corresponding to the terminal vertex in $P_{a,b}$, namely $k[f, g]$. So $P_{\nu}$ and $P_{a,b}$ have the same terminal node, completing the proof.
\end{proof}
 
\begin{rem}
By construction, $P_{a,b}$ exactly tells us where the preimage of the curve $x^b=y^a$ under a sequence of blow ups will have either singularities or non-normal crossing. Since by Lemma \ref{lem: path lemma} we know that $P_{a,b}$ is path, this means that a preimage of $x^b=y^b$ has singularities or non-normal crossings in at most one chart. Note in fact there will be at most one singularity or non-normal crossing, and this will occur at the origin of the chart it occurs in. Moreover, combining Theorem \ref{Thm: resolution rod} with Proposition 3.14 we know that $P_{a,b}$ is finite. More precisely, Theorem \ref{Thm: resolution rod} tells us that $P_{a,b}=P_\nu$ where $\nu$ is the monomial valuation on $k(x,y)$ defined by $\nu(x)=a$ and $\nu(y)=b$, and since $\nu(x)/\nu(y)=a/b$ is rational Proposition 3.14 insures $P_\nu$, and hence $P_{a,b}$ is finite. This means that after a finite number of blow ups, at origin points determined by $P_{a,b}$, the curve $x^b=y^a$ has no singularities and no non-normal intersections, meaning it is resolved. As we will discuss shortly, we actually know precisely how many blow ups are need to resolve $x^b=y^a$ in this way. 
%This allows us to explicitly resolve the singularity of the curve $x^b=y^a$ by blowing up a sequence of points specified by $P_a,b$. In particular $P_a,b$ keeps track of where the preimages of $x^b=y^a$ under blowing ups will have singularities. Thus, by blowing up at these points we manage to resolve any of the singularities and non-normal crossing of curve $x^b=y^a$.
\end{rem}

With this theorem, we have a correspondence between the process of computing the valuation ring for the monomial valuation $\nu$ on $k(x,y)$ (where $\nu(x)=a, \nu(y)=b$ are positive coprime integers) and the process of resolving the curve $V=\V(x^b-y^a)$. In computing $R_{\nu}$, we study a strictly ascending chain of $\nu$-positive subrings of $k(x,y)$, and this chain of subrings is exactly the chain of coordinate rings containing all singularities of $V$ as it is resolved via consecutive blow-ups. Since we have this connection, as well as the previously established connection between $P_{\nu}$ and the continued fraction expansion of $a/b$, there is a natural connection between the process of resolving $V$ and the continued fraction expansion of $a/b$.

\begin{cor}
If $\frac{a}{b}=[d_0;d_1,d_2, \ldots d_n]$, then the curve $\V(x^b-y^a)$ $($where $a,b \in \Z_{>1}$ are coprime$)$ will require $d_0+d_1+\ldots +d_n$ blow-ups to resolve using the previously-described algorithm.
\end{cor}

To see what this corollary means, we apply it to a specific example. 

\begin{exm}

Consider the process of resolving the curve $\V(x^7-y^{24})$. The continued fraction expansion of $24/7$ is $[3;2,3]$. So this curve will take $8$ blow-ups to be completely resolved, and the charts where the singularities will occur are given by the following coordinate rings:
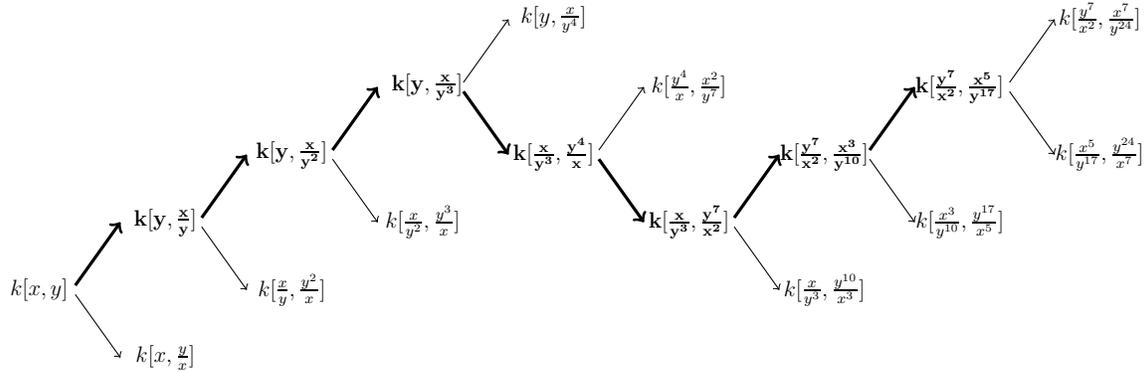
\begin{figure}[H]
\begin{center}
	\resizebox{!}{!}{\begin{tikzpicture}
	[scale=.6,colorstyle/.style={circle, draw=black!100,fill=black!100, thick,
	inner sep=0pt, minimum size=2 mm},every node/.style={scale=.7}]

	\node at (0,0){$k[x,y]$}; 
	
	\draw[very thick,->](.8,.1)--(1.8,1.5);
	\draw[->](.8,-.1)--(1.8,-1.5);
	\node at (2.8,1.5){$\mathbf{k[y, \frac{x}{y}]}$};
	\node at (2.8,-1.5){$k[x, \frac{y}{x}]$};
	
	\draw[very thick,->](3.6,1.6)--(4.6,3);
	\draw[->](3.6,1.4)--(4.6,0);
	
	\node at (5.6,3){$\mathbf{k[y, \frac{x}{y^{2}}]}$};
	\node at (5.6,0){$k[\frac{x}{y}, \frac{y^{2}}{x}]$};
	
	\draw[very thick,->](6.5,3.1)--(7.5,4.5);
	\draw[->](6.5,2.9)--(7.5,1.5);
	
	\node at (8.6,4.5){$\mathbf{k[y, \frac{x}{y^{3}}]}$};
	\node at (8.5,1.5){$k[\frac{x}{y^{2}}, \frac{y^{3}}{x}]$};
	
	\draw[->](9.4,4.6)--(10.4,6);
	\draw[very thick,->](9.4,4.4)--(10.4,3);
	
	\node at (11.4,6){$k[y, \frac{x}{y^{4}}]$};
	\node at (11.4,3){$\mathbf{k[\frac{x}{y^{3}}, \frac{y^{4}}{x}]}$};

	\draw[->](12.4,3.1)--(13.4,4.5);
	\draw[very thick,->](12.4,2.9)--(13.4,1.5);
	
	\node at (14.4, 4.5){$k[\frac{y^{4}}{x}, \frac{x^{2}}{y^{7}}]$};
	\node at (14.4, 1.5){$\mathbf{k[\frac{x}{y^{3}}, \frac{y^{7}}{x^{2}}]}$};
	\draw[very thick,->](15.4,1.6)--(16.4,3);
	\draw[->](15.4,1.4)--(16.4,0);
	
	\node at (17.4, 3){$\mathbf{k[\frac{y^{7}}{x^{2}}, \frac{x^{3}}{y^{10}}]}$};
	\node at (17.4,0){$k[\frac{x}{y^{3}}, \frac{y^{10}}{x^{3}}]$};

	\draw[very thick,->](18.4,3.1)--(19.4,4.5);
	\draw[->](18.4,2.9)--(19.4,1.5);
	\node at (20.4,4.5){$\mathbf{k[\frac{y^{7}}{x^{2}},\frac{x^{5}}{y^{17}}]}$};
	\node at (20.4,1.5){$k[\frac{x^{3}}{y^{10}}, \frac{y^{17}}{x^{5}}]$};
	
	\draw[->](21.5,4.6)--(22.5,6);
	\draw[->](21.5,4.4)--(22.5,3);
	
	\node at (23.5,6){$k[\frac{y^{7}}{x^{2}},\frac{x^{7}}{y^{24}}]$};
	\node at (23.5,3){$k[\frac{x^{5}}{y^{17}}, \frac{y^{24}}{x^{7}}]$};
	\end{tikzpicture}}
	\end{center}
	\caption{Coordinate rings where singularities occur in the resolution of $\V(x^7-y^{24})$.}
\end{figure}
\end{exm}

%So now we see that the continued fraction expansion of $a/b$ gives us enough information to explicitly describe a resolution of $\V(x^b-y^a)$.

\section{Acknowledgments}

This research was partially supported by NSF grant DMS-1001764 and NSF RTG grant DMS-0943832 with the third author supported by a NSF Graduate Research Fellowship under Grant No. PGF-031543. The first and second authors would like to thank the University of Michigan and the Department of Mathematics for hosting them as part of an NSF-REU program at the University of Michigan. The authors would like to thank an anonymous referee whose comments helped improve the readability of this paper. Finally, the authors would also like to thank their advisor Karen E. Smith for her advice and guidance, without which these results would not have been obtained.

\bibliographystyle{plain}	
\bibliography{REU2013}	

\end{document}